\documentclass{article}
\usepackage[utf8]{inputenc}
\usepackage{authblk}

\title{Toughness of recursively partitionable graphs}
\author[1]{Calum~Buchanan}
\author[2]{Brandon~Du~Preez}
\author[3]{K.~E.~Perry}
\author[1]{Puck~Rombach}
\affil[1]{{\small University of Vermont\\
\{calum.buchanan, puck.rombach\}@uvm.edu}}
\affil[2]{{\small University of Cape Town\\
brandondup@gmail.com}}
\affil[3]{{\small Soka University of America\\kperry@soka.edu}}
\date{\today}

\usepackage{graphicx}
\usepackage[a4paper,margin=1.5in,footskip=0.25in]{geometry}
\usepackage {amsmath, amssymb, amsthm, amsfonts}
\usepackage{xcolor}
\usepackage{lineno}
\usepackage{parskip}
\usepackage{tikz}
\usepackage{bm}
\usepackage{float}
\usetikzlibrary{arrows,decorations.pathmorphing,backgrounds,positioning,fit,petri, arrows.meta, matrix, calc, decorations.markings}
\tikzset{>=latex}

\begin{document}
	\newtheorem{theorem}{Theorem}
	\newtheorem{corollary}[theorem]{Corollary}
	\newtheorem{thm}[theorem]{Theorem}
	\newtheorem{lemma}[theorem]{Lemma}
	\newtheorem{obs}[theorem]{Observation}
	\newtheorem{prop}[theorem]{Proposition}
	\newtheorem{conj}[theorem]{Conjecture}
	\newtheorem{dfn}[theorem]{Definition}
	\newtheorem{remark}[theorem]{Remark}
	\newtheorem{prob}{Problem}
	\newtheorem{claim}[theorem]{Claim}
	\newtheorem{note}[theorem]{Note}
	\newtheoremstyle{case}{}{}{}{}{}{:}{ }{}
	\theoremstyle{case}
	\newcommand*\midpoint[1]{\overline{#1}}
	\newtheorem{case}{Case}
	
\maketitle
	
\begin{abstract}
	\noindent A simple graph $G=(V,E)$ on $n$ vertices is said to be \emph{recursively partitionable (RP)} if $G\simeq K_1$, or if $G$ is connected and satisfies the following recursive property: for every integer partition $a_1, a_2, \dots, a_k$ of $n$, there is a partition $\{A_1, A_2, \dots, A_k\}$ of $V$ such that each $|A_i|=a_i$, and each induced subgraph $G[A_i]$ is RP ($1\leq i \leq k$).
	We show that if $S$ is a vertex cut of an RP graph $G$ with $|S|\geq 2$, then $G-S$ has at most $3|S|-1$ components.
	Moreover, this bound is sharp for $|S|=3$.
	We present two methods for constructing new RP graphs from old. We use these methods to show that for all positive integers $s$, there exist infinitely many RP graphs with an $s$-vertex cut whose removal leaves $2s+1$ components.
	Additionally, we prove a simple necessary condition for a graph to have an RP spanning tree, and we characterise a class of minimal 2-connected RP graphs.
	\end{abstract}

\section{Introduction}
Let $n$ be a positive integer. 
An \textbf{integer partition} of $n$ is a list $a_1, \dots, a_k$ of positive integers such that 
$a_1 \leq a_2 \leq \dots \leq a_k$ and 
$a_1 + \dots + a_k = n$.
Let $G=(V,E)$ be a graph of order $n$.
An \bm{$(a_1, ..., a_k)$}\textbf{-partition} of $G$ is a partition $\{A_1, \dots, A_k\}$ of $V$ such that $|A_i| = a_i$ for all $i$.
We say the partition has \textbf{connected parts} if, for all $i\in \{1,\dots, k\}$, the induced subgraphs $G[A_i]$ are connected.

In 1976, Gy\"{o}ri and Lov\'{a}sz considered the problem of determining when a graph has an $(a_1,\dots, a_k)$-partition with connected parts and independently proved the following theorem.

\begin{thm}[Gy\"{o}ri-Lov\'{a}sz \cite{divisionsubgraphs_gyori, kpartition_lovasz}] 
	Let $G$ be a graph of order $n$ and $a_1, \dots, a_k$ an integer partition of $n$. 
	If $G$ is $k$-connected, then it has an $(a_1, \dots, a_k)$-partition with connected parts.
	\label{thm_gyori_lovasz}
\end{thm}
	
We say $G$ is \textbf{arbitrarily partitionable} (or just \textbf{AP}) if, for every integer partition $a_1, \dots, a_k$ of $n$, there exists an $(a_1, \dots, a_k)$-partition of $V$ with connected parts.
AP graphs were introduced in~\cite{PolytimeAPtripode_barth},
and a polynomial time 
algorithm for determining whether a subdivision of $K_{1,3}$ is AP was provided.

The graph $G$ is \textbf{recursively partitionable} (\textbf{RP}) if $G \simeq K_1$, or $G$ is connected and satisfies the following recursive property: 
for every integer partition $a_1, \dots, a_k$ of $n$, there is an $(a_1, \dots, a_k)$-partition $\{A_1, \dots, A_k\}$ of $V$ such that each $G[A_i]$ is RP.
RP graphs were introduced in~\cite{ravdsuns_baudon, ravdgraphs_baudon}.

In~\cite{ravdgraphs_baudon}, RP trees were characterised (among other results), and in~\cite{ravdsuns_baudon}, a class of RP unicyclic graphs was characterised.
In both papers, the authors made heavy use of the following characterisation of RP graphs.

\begin{prop} \textup{\cite{ravdsuns_baudon}}
	An $n$-vertex graph $G=(V,E)$ is RP if and only if it is connected, and:
	\begin{itemize}
		\item $G\simeq K_1$, or
		\item for every partition $a,b$ of $n$, there is an $(a,b)$-partition $\{A,B\}$ of $V$ such that both $G[A]$ and $G[B]$ are RP.
	\end{itemize}
	\label{prop_simple_RP_char}
\end{prop}

RP graphs were independently introduced (as ``partition wonderful graphs") as a result of investigations into rainbow-cycle-free edge colorings (such as in \cite{rainbow_forbidden_hoffman}), by
Peter Johnson, with the help of Paul Horn, at the MASAMU 2020 workshop.
 
These graphs arise naturally when considering rainbow-cycle-free edge colorings (which are of recent interest in their own right: \cite{gerbner2022jl, janzer2022, kevash2007jl}.)
A \textbf{JL-coloring} of an $n$-vertex graph is an edge coloring using exactly $n-1$ colors that does not contain any rainbow cycles. 
These colorings are studied for $K_n$ in~\cite{devilbiss2022jl} and~\cite{edgecol_gouge}, $K_{n,m}$ in~\cite{johnson2017jl} and complete multipartite graphs in~\cite{johnson20172jl}.

In~\cite{rainbow_forbidden_hoffman},
the authors introduced the following \textbf{standard construction} for creating a JL-coloring of a connected graph $G$:
\begin{enumerate}
	\item If $n>1$, find a partition $V=\{A,B\}$ with connected parts,
	\item color edges between $A$ and $B$ with a single color that will not be used again,
	\item iterate (1) and (2) on $G[A]$ and $G[B]$.
\end{enumerate}
This leads to the main result of \cite{rainbow_forbidden_hoffman}:
\begin{thm}\textup{\cite{rainbow_forbidden_hoffman}}
	Every JL-coloring is obtainable by an instance of the standard construction.
	\label{thm_standard_constrution}
\end{thm}

\begin{corollary} \textup{\cite{rainbow_forbidden_hoffman}}
	Every JL-coloring of a connected graph $G=(V,E)$ is the restriction of a JL-coloring of the complete graph with vertex set $V$.
	\label{cor_jlcol_restriction}
\end{corollary}

Combining Proposition \ref{prop_simple_RP_char}, Theorem \ref{thm_standard_constrution} and Corollary \ref{cor_jlcol_restriction} yields the following observation of Johnson:

\begin{obs}
	A connected graph $G=(V,E)$ of order $n$ is RP if and only if every JL-coloring $\varphi$ of $K_n$ can be restricted to a JL-coloring $\varphi|_E$ of a copy of $G$.
	\label{obs_jlcolor_rp_char}
\end{obs}

The rest of this paper is organised as follows. 
In Section~\ref{sec:defs}, we define useful graph-theoretical tools and constructions that will be used throughout the paper. 
In Section~\ref{sec:elresults}, we list basic observations about the properties of AP and RP graphs. 
In Section~\ref{sec:build}, we introduce recursive constructions of RP graphs, which we later use to find infinite classes of RP graphs with a given toughness. 
It is easy to see that if a graph has an AP (RP) spanning tree, then it is AP (RP). 
In Section~\ref{sec:spanning}, we take a more detailed look at spanning subgraphs of RP graphs and provide a necessary condition for an RP graph to have a spanning tree homeomorphic to $K_{1,k}$. 
We also show that if an RP graph has an RP spanning tree, then for every $S\subseteq V$ we have $c(G-S)\leq |S|+2$. 
In Section~\ref{sec:lowbounds}, we find lower bounds for the maximum possible values of $c(G-S)$ for $S\subseteq V$ in an RP graph $G$. 
In particular, we show that, for any $s$, there exists an infinite family of RP graphs with a $s$-vertex cut whose removal leaves $2s+1$ components. 
In Section~\ref{sec:minimal}, we show that there exists a finite set of minimal RP graphs for any given possible cut size $|S|$ and $c(G-S)$.
In Section~\ref{sec:upbound}, we bound $c(G-S)$ from above, by showing that in an RP graph $G$, for any $S \subseteq V$, we have $c(G-S)\leq 3 |S|-1$, which shows that every RP graph is $\tfrac{1}{3}$-tough. 
Finally, in Section~\ref{sec:quests}, we list a set of open questions.

	\section{Additional definitions}\label{sec:defs}
	\subsection{Properties and parameters}
	
	For a positive integer $k$, let $\bm{E_k}$ denote the \textbf{empty graph} with $k$ vertices and no edges.
	If $G$ is a graph, then $\bm{n(G)}$ is its \textbf{order} (number of vertices), $\bm{m(G)}$ its number of edges. Let $\bm{\alpha(G)}$ denote its \textbf{independence number} (the order $k$ of a maximum induced $E_k$ subgraph), and $\bm{\kappa(G)}$ its \textbf{vertex connectivity} (the minimum cardinality of a set of vertices whose removal disconnects $G$).
	Let \[\bm{\sigma(G)} = \min\{d(u) + d(v) : \text{$u$ and $v$ are non-adjacent vertices of $G$}\}.\]

	A graph is \textbf{traceable} if it has a spanning path (i.e., a Hamiltonian Path) and \textbf{Hamiltonian} if it has a spanning cycle (i.e., a Hamiltonian cycle).
	
	A \textbf{perfect matching} of a graph is a set $M$ of edges that are all pairwise disjoint, such that every vertex is incident with an edge in $M$. 
	A \textbf{near-perfect matching} is a set $M$ of edges that are all pairwise disjoint, such that every vertex except for one is incident with an edge in $M$.
	A graph is \textbf{(near)} \textbf{matchable} if it has a (near) perfect matching.
	
	Let $G_1 = (V_1, E_1), \dots, G_n = (V_n, E_n)$ be graphs. 
	The \textbf{sequential join} $G_1 + \dots + G_n$ is the graph formed by taking the graph union $(V_1\cup\dots\cup V_n, E_1\cup\dots\cup E_n)$ and adding to it all edges of the form $uv$, where $u\in V_i$ and $v\in V_{i+1}$ ($1\leq i < n$).
	
	Let $G=(V,E)$ be a graph.
	Denote by \bm{$c(G)$} the number of components of $G$.
	In particular, for a connected graph $G$, if $S\subseteq V$, then $c(G-S) \geq 2$ if and only if $S$ is a cut. 
	The \textbf{toughness} \bm{$\tau(G)$} of $G$ is
	\[
		\tau(G) = \min\left\{ \frac{|S|}{c(G-S)} : S\subseteq V, c(G-S) \geq 2 \right\}.
	\]
	For a positive real number $r$, we say $G$ is \textbf{$r$-tough} if $\tau(G)\geq r$.

	\subsection{Graph constructions}
	
	In this section, we define graph constructions that we will use throughout the paper. See Figures~\ref{fig_tree_examples} and~\ref{fig_pw_no_pw_tree} for examples.
	
	Let $a,b,c$ be positive integers. 
	The \textbf{tripode} graph $T(a,b,c)$ is the tree that has one degree 3 vertex, $v$, the removal of which leaves three paths having $a$, $b$ and $c$ vertices. 
	
	Let $k$ be a positive integer, and $b_i$, $1\leq i\leq k$ be non-negative integers. 
	The \textbf{balloon} graph $B(b_1, b_2, \dots, b_k)$ consists of $k$ paths $P_i$ and two non-adjacent vertices $u$ and $v$. Further, $n(P_i) = b_i$, the first vertex of $P_i$ is adjacent to $u$, and the last vertex of $P_i$ is adjacent to $v$.

	Let $k$ be a positive integer, and let $b_0, b_1, \dots, b_k$ be non-negative integers. 
	The \textbf{semistar} $K_{b_0}(b_1, b_2, \dots, b_k)$ is the graph formed from the disjoint union of (possibly null) cliques $K_{b_0}$, $K_{b_1}$, $\dots$, $K_{b_k}$ by adding every possible edge between a vertex of $K_{b_0}$ and a vertex not belonging to $K_{b_0}$. Symbolically, 
	\[
		K_{b_0}(b_1, \dots, b_k) = K_{b_0} + \left( \bigcup_{i=1}^k K_{b_i} \right).
	\]
	Note that $K_{b}(0,\dots, 0) \simeq K_0(0,\dots, b, \dots, 0)\simeq K_b$, and that $K_{b_0}(b_1,\dots,b_k,0) \simeq \\ K_{b_0}(b_1,\dots,b_k)$.

	Suppose $K_{b_0}(b_1, \dots, b_k)$ is RP, and $\{G_i\}_{i=0} ^k$ is a set of RP graphs such that $n(G_i) = b_i$. A graph $H$ is a \textbf{replacement graph} for $K_{b_0}(b_1, \dots, b_k)$ with respect to $\{G_i\}_{i=0} ^k$ if
	\[
		H = G_0 + \left( \bigcup_{i=1}^k G_i \right).
	\]
	
	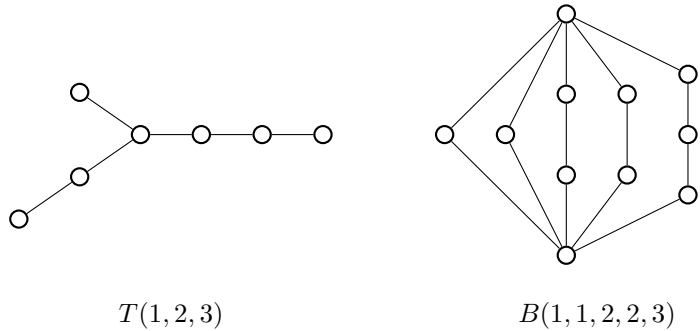
\begin{figure}[h]
	\begin{center}
	\begin{tikzpicture}
		[scale=0.8,inner sep=0.8mm, 
		vertex/.style={circle,thick,draw}, 
		thickedge/.style={line width=2pt}] 
		
		\begin{scope}[shift={(0,0)}]
			\node[vertex] (0) at (0,0) {};
			\node[vertex] (1) at (-1,0.7) {};
			\node[vertex] (2) at (-1,-0.7) {};
			\node[vertex] (2') at (-2,-1.4) {};
			\node[vertex] (3) at (1,0) {};
			\node[vertex] (4) at (2,0) {};
			\node[vertex] (5) at (3,0) {};
			
			\draw (2')--(2)--(0)--(1) (0)--(3)--(4)--(5);
			
			\node at (0.5,-3) {$T(1,2,3)$};
		\end{scope}
		
		\begin{scope}[shift={(7,0)}]
			\node[vertex] (a) at (0,2) {};
			\node[vertex] (b) at (0,-2) {};
			\node[vertex] (1) at (-2,0) {};
			\node[vertex] (2) at (-1,0) {};
			\node[vertex] (3) at (0,0.67) {};
			\node[vertex] (4) at (0,-0.67) {};
			\node[vertex] (5) at (1,0.67) {};
			\node[vertex] (6) at (1,-0.67) {};
			\node[vertex] (7) at (2,1) {};
			\node[vertex] (8) at (2,0) {};
			\node[vertex] (9) at (2,-1) {};
			
			\draw (a)--(1)--(b);
			\draw (a)--(2)--(b);
			\draw (a)--(3)--(4)--(b);
			\draw (a)--(5)--(6)--(b);
			\draw (a)--(7)--(8)--(9)--(b);
			
			\node at (0.5,-3) {$B(1,1,2,2,3)$};
		\end{scope}
		
	\end{tikzpicture}
	\end{center}
	\caption{The tripode $T(1,2,3)$ and the balloon $B(1,1,2,2,3)$.}
	\label{fig_tree_examples}
	\end{figure}

	\begin{figure}[h]
	\begin{center}
	\begin{tikzpicture}
		[scale=0.8,inner sep=0.6mm, 
		vertex/.style={circle,thick,draw}, 
		thickedge/.style={line width=2pt}] 
		
		\begin{scope}[shift={(0,0)}]
			\node[vertex] (x) at (-0.5,0) {};
			\node[vertex] (y) at (0.5,0) {};
			
			\node[vertex] (1) at (-4,3) {};
			\node[vertex] (2) at (-3,3) {};
			\node[vertex] (3) at (-2,3) {};
			
			\node[vertex] (4) at (-1,3) {};
			\node[vertex] (5) at (-0,3) {};
			
			\node[vertex] (6) at (1,2.7) {};
			\node[vertex] (7) at (2,3.4) {};
			\node[vertex] (8) at (3,3.4) {};
			\node[vertex] (9) at (4,2.7) {};
			
			\foreach \i in {1,...,9}
				{
					\draw (x)--(\i)--(y);
				}
			\draw (4)--(5);
			\draw (6)--(7)--(8)--(9)--(6) (7)--(9) (6)--(8);
			\draw (x)--(y);
		\end{scope}

	\end{tikzpicture}
	\end{center}
	\caption{The semistar $K_2(1,1,1,2,4)$.}
	\label{fig_pw_no_pw_tree}
	\end{figure}
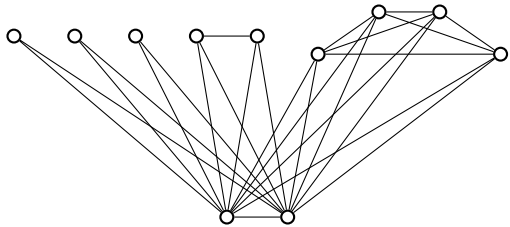

	\section{Elementary and known results}\label{sec:elresults}
	
	In this section, we list a number of useful literature results on AP and RP graphs.
	We make frequent use of these results and observations, particularly Lemma \ref{lem_spanning_subgraph}, Theorem \ref{thm:pw_tree_characterisation} and Observations \ref{obs_tripode_balloon_universal} and \ref{obs_K_universal}.
	We also present a characterisation of AP and RP complete multipartite graphs.
	
	\begin{lemma}\textup{\cite{ravdgraphs_baudon}}
		If a graph $G$ has an RP (AP) spanning subgraph, then $G$ is itself RP (AP).
		\label{lem_spanning_subgraph}	
	\end{lemma}
	
	\begin{obs}\textup{\cite{ravdgraphs_baudon}}
		Let $G$ be a graph. 
		The following implications for properties of $G$ hold, and none of their converses hold:
		\[
			\text{traceable} \implies \text{RP} \implies \text{AP} \implies \text{(near) matchable}.
		\]

		\label{obs_pw_is_pi_is_pg}
	\end{obs}
	
	The following lemma, by
	Bondy and Chvatal~\cite{methodstability_bondychvatal} is a somewhat well-known variation of Ore's Hamiltonicity Theorem~\cite{hamiltonian_ore}. 
	
	\begin{lemma}\textup{\cite{methodstability_bondychvatal}}
		Let $G$ be a graph of order $n$. 
		If $\sigma(G) \geq n-1$, then $G$ is traceable.
		\label{lem_ore_traceable}
	\end{lemma}

	\begin{thm}[Ore's Theorem \cite{hamiltonian_ore}]
		Let $G$ be a graph of order $n$. 
		If $\sigma(G) \geq n$, then $G$ is Hamiltonian.
		\label{thm_ore_hamiltonian}
	\end{thm}
	
	With Lemma \ref{lem_ore_traceable} we easily prove the following.
	
	\begin{prop}
		Let $G$ be a graph with $\sigma(G) \geq 2k$ and order $n$.
		If $n\leq 2k+1$, then $G$ is RP (and therefore AP), and this bound is sharp.
		\label{thm_pw_order_bound}
	\end{prop}
	
	\begin{proof}
		The graph $G$ is RP since it is traceable (Observation \ref{obs_pw_is_pi_is_pg} and Lemma \ref{lem_ore_traceable}).
		
		To prove the bound is sharp, consider the complete bipartite graph $K_{k, k+2}$. 
		This graph has $\sigma = 2k$ and order $2k+2$.
		However $K_{k, k+2}$ does not have a perfect matching, and thus by Observation \ref{obs_pw_is_pi_is_pg}, it is not RP.
	\end{proof}

In~\cite{ore_marczyk}, Marczyk showed that the above result can be improved for AP graphs with the extra condition $\alpha(G) \leq \lceil \frac{n(G)}{2} \rceil$.
	
	\begin{thm}\textup{\cite{ore_marczyk}}
		Let $G$ be a connected graph of order $n$. 
		If $\alpha(G) \leq \lceil \frac{n}{2} \rceil$ and $\sigma(G) \geq n-3$, then $G$ is AP.
	\end{thm}
	
	For $G$ to have a (near) perfect matching, it is clearly necessary that $\alpha(G) \leq \lceil \frac{n(G)}{2} \rceil$. 
	For a large class of graphs, including complete multipartite graphs, this condition is also sufficient. We summarize these equivalences in Proposition~\ref{prop:multipartequivalences}.

Note that there is no possible forbidden subgraph characterisation of AP (RP) graphs. Given any graph $G$ of order $n$, the graph $K_n + G$ is Hamiltonian, and thus AP (RP).
	
	\begin{prop}\label{prop:multipartequivalences}
		Suppose $G$ is a graph of order $n$ such that $K_{a,b} \leq G \leq K_a + E_b$ for $a\leq b$ positive integers, or that $G$ is a complete multipartite graph.
		The following are equivalent:
		\begin{itemize}
			\item[(i)] $\alpha(G) \leq \lceil \frac{n}{2} \rceil$,
			\item[(ii)] $G$ has a (near) perfect matching,
			\item[(iii)] $G$ is traceable, 
			\item[(iv)] $G$ is AP,
			\item[(v)] $G$ is RP.
		\end{itemize}
		\label{prop_char_splitpartite}
	\end{prop}
	
\begin{proof}
 It is easy to verify that \emph{(iii)} implies \emph{(ii)} and that \emph{(ii)} implies \emph{(i)}. We now argue that \emph{(i)} implies \emph{(iii)}.
  If $G$ is complete multipartite or $K_{a,b} \leq G \leq K_a + E_b$, and $\alpha(G) \leq \lceil \frac{n}{2} \rceil$, then the minimum degree satisfies $\delta(G) \geq \lfloor \frac{n}{2} \rfloor$. 
 Consider the join $G+\{v\}$ and note that $\delta(G+\{v\}) \geq \frac{n+1}{2}$.
 By Ore's Theorem (Theorem \ref{thm_ore_hamiltonian}), $G+\{v\}$ has a spanning cycle $C$, so $C-v$ is a spanning path of $G$.
 Observation~\ref{obs_pw_is_pi_is_pg} completes the proof.
\end{proof}

In~\cite{ravdgraphs_baudon}, Baudon, Gilbert and Woźniak characterised RP trees.
In~\cite{structavd_baudon}, Baudon, Bensmail, Foucaud and Pilsniak described some properties of RP balloons.

\begin{thm}\textup{\cite{ravdgraphs_baudon}}
	A tree is RP if and only if it is either a path, the tripode $T(2,4,6)$, or a tripode $T(a,b,c)$, where $(a,b,c)$ is one of the triples in Table~\ref{tab_Tabc_RP}:
	\begin{table}[h]
		\centering
		\begin{tabular}{|l|l|l|l|l|}
			\hline
			$(1,1,c)$ & $c\equiv 0 \pmod 2$                        &  & $(1,4,c)$ & $c \in \{5, 6, 8, 10, 13, 18\}$ \\ \hline
			$(1,2,c)$ & $c\equiv 0 \pmod 3$ or $c\equiv 1 \pmod 3$ &  & $(1,5,6)$ &                                 \\ \hline
			$(1,3,c)$ & $c\equiv 0 \pmod 2$                        &  & $(1,6,c)$ & $c\in \{7, 8, 10, 12, 14\}$     \\ \hline
		\end{tabular}
	\caption{Table of triples $(1,b,c)$, $1 \leq b\leq c$, for which the graph $T(a,b,c)$ is RP.}
	\label{tab_Tabc_RP}
	\end{table}
	\label{thm:pw_tree_characterisation}
\end{thm}

\begin{thm}\textup{\cite{structavd_baudon}}
	Let $B$ be the balloon graph $B(b_1,\dots, b_k)$ with $b_1 \leq \dots \leq b_k$.
	If $B$ is RP, then $k\leq 5$.
	Further, if $B$ is RP and $k\geq 4$, then $b_1 \leq 7$ and $b_2 \leq 39$, but $b_k$ can be arbitrarily large.
	\label{thm_rp_balloon_toughness}
\end{thm}

Tripodes and balloons are ``universal'' for RP graphs with connectivity 1 and 2, respectively, as the following observation from \cite{avdconnectivity_baudon} shows.

\begin{obs}\textup{\cite{avdconnectivity_baudon}}
	Let $S$ be a vertex cut of a connected graph $G$, let $C_1,\dots, C_k$ denote the components of $G-S$, and let $c_i$ denote $n(C_i)$.
	\begin{itemize}
		\item If $|S|=1$ and $G$ is RP (AP), then the tripode $T(c_1, \dots, c_k)$ is RP (AP),
		\item If $|S|=2$ and $G$ is RP (AP), then the balloon $B(c_1, \dots, c_k)$ is RP (AP).
	\end{itemize}
	\label{obs_tripode_balloon_universal}
\end{obs}

To discuss AP and RP graphs of arbitrary connectivity, we find it easiest to work with the semistars $K_{b_0}(b_1,\dots, b_k)$, as they are also ``universal''.

\begin{obs}
	Let $S$ be an $s$-vertex cut of a graph $G$, let $C_1,\dots, C_k$ denote the components of $G-S$, and let $c_i$ denote $n(C_i)$.
	If $G$ is RP (AP), then the semistar $K_s(c_1,\dots, c_k)$ is RP (AP).
	\label{obs_K_universal}
\end{obs}

\begin{proof}
	Notice that $G$ is a spanning subgraph of $K_s(c_1,\dots, c_k)$ and apply Lemma \ref{lem_spanning_subgraph}.
\end{proof}

Per Observations \ref{obs_tripode_balloon_universal} and \ref{obs_K_universal}, the triples $(a,b,c)$ in Table \ref{tab_Tabc_RP} for which $T(a,b,c)$ is RP are also the triples for which $K_1(a,b,c)$ is RP.

\section{New RP graphs from old}\label{sec:build}

In this section, we present two operations for combining RP graphs to obtain new RP graphs: the well-known sequential join, and a ``subgraph replacement" operation. 
These constructions, in tandem with Lemma \ref{lem_spanning_subgraph}, allow us to easily prove that many graphs encountered in the rest of the paper are RP.
Of particular interest is the use of replacement graphs in Section~\ref{sec:lowbounds} to construct RP graphs with large vertex cuts that leave many components.

There is a generalisation of the fact that paths are RP. 
In particular, the sequential join of RP graphs is RP.

\begin{prop}
	Let $H_1, \dots, H_k$ be RP graphs, and let $G = H_1 + \dots + H_k$ be the sequential join of the graphs $H_i$. 
	Then $G$ is RP.
	\label{prop_sequential_join}
\end{prop}

\begin{proof}
	Let $n_i$ be the order of $H_i$, and $n = n_1 + \dots + n_k$ the order of $G$. We proceed by induction on $n$.
	
	The base case $n=1$ is trivial, as then $G \simeq K_1$, which is RP.
	
	Let $n\geq 2$, assume the proposition is true for all positive integers less than $n$, and let $G$ be an $n$-vertex sequential join of RP graphs $H_1, \dots, H_k$. 
	It suffices to show that for any $a\in [1, n-1]$, there exists a partition of $G$ into two RP graphs $G[A]$ and $G[B]$ such that $G[A]$ has order $a$. 
	To do this, we will pick the subgraph induced by the `leftmost' $a$ vertices of $G$ in a manner that breaks apart at most one of the graphs $H_i$.
	
	Let $m_0 = 0$, and for all $i\in [1,k]$, let $m_i = \sum_{j=1}^{i}n_j$. 
	Denote by $s$ the largest non-negative integer such that $a\geq m_s$.
	Since the graph $H_{s+1}$ is RP, it can be partitioned into two RP parts $H_{s+1}[X]$ and $H_{s+1}[Y]$, such that $|X|= a - m_s$.
	We can thus pick $A = V(H_1)\cup \dots \cup V(H_s)\cup X$ and $B = Y \cup V(H_{s+2})\cup \dots \cup V(H_k)$. 
	Note that $G[A] = H_1 + \dots + H_s + H_{s+1}[X]$ and $G[B] = H_{s+1}[Y] + H_{s+2} + \dots + H_k$. 
	By the induction hypothesis, both $G[A]$ and $G[B]$ are RP, completing the proof.
\end{proof}

A consequence of Proposition \ref{prop_sequential_join} is that the suspension $K_1 + G$ of an RP graph $G$ is RP.

\begin{corollary}
	Suppose $t$ is a positive integer. If $K_{a_0}(a_1, \dots, a_k)$ and $K_{b_0}(b_1, \dots, b_j)$ are RP, then so is the graph
	\[
		K_{a_0+b_0+t}(a_1, \dots, a_k, b_1, \dots, b_j).
	\]
	\label{cor_ka_plus_kb_plus_1}
\end{corollary}

\begin{proof}
	The graph $K_{a_0+b_0+t}(a_1, \dots, a_k, b_1, \dots, b_j)$ has a spanning subgraph isomorphic to
	\[
		K_{a_0}(a_1, \dots, a_k) + K_t + K_{b_0}(b_1, \dots, b_j),
	\]
	which is RP by Proposition \ref{prop_sequential_join}.
\end{proof}

Let $\{G_i\}_{i=1}^k$, $k\geq 2$ be a set of graphs, $J = G_1 + \dots + G_k$, and $m = \min\{\tau(G_i) : 1\leq i \leq k\}$ be the minimum toughness among the graphs $G_i$.
If $m < \frac{1}{2}$, then $\tau(J) > m$. 
Thus, there are limitations to how low the toughness of a sequential join of RP graphs can be. 
However, replacement graphs provide RP graphs with high connectivity and low toughness (see Corollary \ref{cor_2kplus1_family}).

\begin{theorem}(RP Replacement Theorem)
	Every replacement graph is RP. 
	That is, suppose $K_{b_0}(b_1, \dots, b_k)$ is RP, and $\{G_i\}_{i=0}^k$ is a set of RP graphs such that $n(G_i) = b_i$. 
	Then the graph $H$ is RP, where
	\[
		H = G_0 + \left(\bigcup_{i=1}^k G_i \right).
	\]
	\label{thm_replacement}
\end{theorem}

\begin{proof}
	We use induction on the order of the replacement graph.
	Clearly every replacement graph of order at most 3 is RP.
	Suppose that every replacement graph of order $n-1$ or less is RP, and let $H= G_0 + \left(\bigcup_{i=1}^k G_i \right)$ be a replacement graph of order $n$.
	Denote $n(G_i) = b_i$ and $K = K_{b_0}(b_1, \dots, b_k)$.
	Let $\lambda$ be any positive integer such that $1\leq \lambda < n$.
	It suffices to prove that there is a partition $V(H) = \{X', Y'\}$ such that $|X'| = \lambda$ and $H[X']$, $H[Y']$ are both RP.
	
	Since $H$ is a replacement graph, $K$ is RP by definition.
	Thus, there is a partition $V(K) = \{X,Y\}$ such that $|X| = \lambda$, and the induced subgraphs $K[X] = K_{x_0}(x_1, \dots, x_k)$ and $K[Y] = K_{y_0}(y_1, \dots, y_k)$ are RP.
	Note that $x_i + y_i = b_i$, and that we may have $x_i = 0$ ($y_i = 0$) for some $i$.
	Since $G_i$ is RP, it has a partition $V(G_i) = \{X_i, Y_i\}$ with $|X_i| = x_i$ and $|Y_i| = y_i$ such that $G_i[X_i]$ and $G_i[Y_i]$ are RP.
	
	Thus, let
	\[
		X' = \bigcup_{i=0}^k X_i \;\;\;\text{and}\;\;\; Y' = \bigcup_{i=0}^k Y_i.
	\]
	Note that $\{X',Y'\}$ is a partition of $V(H)$ such that $|X'| = x_0 + \dots + x_k = |X| = \lambda$.
	Further, both $H[X']$ and $H[Y']$ are replacement graphs:
	\[
		H[X'] = G_0[X_0] + \left(\bigcup_{i=1}^k G_i[X_i] \right) \;\;\;\text{and}\;\;\; H[Y'] = G_0[Y_0] + \left(\bigcup_{i=1}^k G_i[Y_i] \right).
	\]
	By the induction hypothesis, both $H[X']$ and $H[Y']$ are RP, so $H$ is RP.
\end{proof}

\section{RP spanning subgraphs}\label{sec:spanning}

It is clear that every graph with an RP (AP) spanning tree is RP (AP).
In \cite{ravdgraphs_baudon}, it was shown that an AP graph need not have an AP spanning tree.
Using the sequential join, it is easy to construct RP graphs that do not have a spanning tripode $T(a,b,c)$ (and thus do not have an RP spanning tree).
For example, the graph $T(1,1,2) + K_1 + T(1,1,2)$ is RP but has no spanning tripode.
In this section, we give a necessary condition for a graph to have a spanning tree homeomorphic to $K_{1,k}$ ($k\in \mathbb{N}$).

\begin{thm}
	Let $G$ be a graph, $k\geq 2$ a positive integer and $S$ a subset of $V(G)$. 
	If $c(G-S) \geq |S| + k$, then $G$ does not have a spanning subdivision of $K_{1,k}$.
	\label{thm_no_spanning_spider}
\end{thm}

\begin{proof}
	Let $c(G-S) = c$ and $|S| = s$. 
	Let $G_1, G_2, \dots, G_c$ denote the components of $G-S$.
	Assume contrary to the theorem statement that there is a subdivision $T$ of $K_{1,k}$ spanning $G$, and that $c-s \geq k$.
	Denote by $v$ the vertex of $T$ such that $d_T(v) = k$, and let $P_1, P_2, \dots, P_k$ be the $k$ maximal paths of $T-v$.
	There are two cases to consider. 
	In both cases, we count the number of components $G_i$ that each path $P_j$ intersects.
	
	\textit{Case 1:} $v\notin S$.\\
	Assume without loss of generality that $v\in G_1$.
	Let $\zeta(P_i) = |\{j\geq 2 : V(G_j)\cap V(P_i) \neq \emptyset\}|$ be the number of components (other than $G_1$) that contain a vertex of $P_i$.
	Between any two vertices of $P_i$ that lie in different components of $G-S$, there must be a vertex of $S$.
	Therefore, for all $i$, we have
	\begin{align}
		\zeta(P_i) \leq |S\cap V(P_i)|.
		\label{eq_path_1}
	\end{align}
	Each of the components $G_2, G_3, \dots, G_c$ must intersect at least one path $P_i$, so
	\begin{align}
		c-1 \leq \sum_{i=1}^{k} \zeta(P_i).
		\label{eq_path_2}
	\end{align}
	Since the paths $P_i$ are disjoint, we have 
	\begin{align}
		\sum_{i=1}^{k} |S\cap V(P_i)| \leq |S| = s.
		\label{eq_path_3}
	\end{align}
	Combining Inequalities \ref{eq_path_1}, \ref{eq_path_2} and \ref{eq_path_3}, we obtain the following inequality
	\begin{align*}
		c-1 \leq \sum_{i=1}^{k} \zeta(P_i) \leq \sum_{i=1}^{k} |S\cap V(P_i)| \leq s.
	\end{align*}
	But this contradicts the fact that $c-s \geq k \geq 2$.
	
	\textit{Case 2:} $v\in S$\\
	Let $\eta(P_i) = |\{j : V(G_j)\cap V(P_i) \neq \emptyset\}|$ be the number of components that contain a vertex of $P_i$.
	Between any two vertices of $P_i$ from different components of $G-S$, there must be a vertex of $S$, however, it is possible that the end vertices of $P_i$ are not in $S$. Thus, for all $i$, the following inequality holds
	\begin{align}
		\eta(P_i) \leq |S\cap V(P_i)| + 1.
		\label{eq_bpath_1}
	\end{align}
	Since every component $G_i$ intersects at least one path $P_i$:
	\begin{align}
		c \leq \sum_{i=1}^{k} \eta(P_i).
		\label{eq_bpath_2}
	\end{align}
	Since the paths $P_i$ are disjoint, and none contain the vertex $v\in S$, we have
	\begin{align}
		\sum_{i=1}^{k} |S\cap V(P_i)| \leq |S-\{v\}| = s-1.
		\label{eq_bpath_3}
	\end{align}
	Putting Inequalities \ref{eq_bpath_1}, \ref{eq_bpath_2} and \ref{eq_bpath_3} together, we obtain
	\begin{align*}
		c \leq \sum_{i=1}^{k} \eta(P_i) \leq \sum_{i=1}^{k} |S\cap V(P_i)| + k \leq s - 1 + k.
	\end{align*}
	But this contradicts the fact that $c\geq s+k$.
\end{proof}

By Theorem \ref{thm:pw_tree_characterisation}, every RP tree on at least 3 vertices is either a path (subdivided $K_{1,2}$) or a subdivided $K_{1,3}$. 
Further, every RP balloon is spanned by a subdivided $K_{1,k}$ with $k\leq 5$, per Theorem \ref{thm_rp_balloon_toughness}.
Thus, we have the following corollary.

\begin{corollary}
	If $G=(V,E)$ contains an RP spanning tree, then every $S\subset V$ satisfies $c(G-S) \leq |S| + 2$.
	If $G$ is spanned by an RP balloon, then every $S\subset V$ satisfies $c(G-S) \leq |S| + 4$.
	\label{cor_c_minus_s_pw}
\end{corollary}

\section{Bounding \bm{$c(G-S)$} from below}\label{sec:lowbounds}

Let $G$ be an RP graph and $S\subseteq V(G)$.
Per Theorem \ref{thm:pw_tree_characterisation} and Observation \ref{obs_tripode_balloon_universal}, if $|S|=1$, then $c(G-S) \leq 3$, and the infinite family of RP tripodes $\{T(1,1,2k)\}_{k\in\mathbb{N}}$ all achieve this bound.
Theorem \ref{thm_rp_balloon_toughness} and Observation \ref{obs_tripode_balloon_universal} show that if $|S|=2$, then $c(G-S) \leq 5$, and the RP balloons $\{B(1,1,2,3,2k)\}_{k\in\mathbb{N}}$ achieve this bound \cite{ravdgraphs_baudon}.
In this section, we bound the maximum possible value of $c(G-S)$ from below.
In particular, we show that for all $s$, there are infinitely many RP graphs with an $s$-vertex cut $S$ such that $c(G-S) = 2s+1$.
Further, we prove that there exists an RP graph $G$ with a cut $S$ such that $|S|=3$ and $c(G-S) = 8$.

\begin{lemma}
	The following graphs are RP:
	\begin{itemize}
		\item[(i)] $K_1(a,b,c)$ for $(a,b,c)=(2,4,6)$ and all $(a,b,c)$ in Table \ref{tab_Tabc_RP},
		\item[(ii)] $K_2(a,b,c,d)$ for $(a,b,c) = (2,4,6)$ and all $(a,b,c)$ in Table~\ref{tab_Tabc_RP}, and for all $d \in \mathbb{N}$,
		\item[(iii)] $K_0(0,\dots,d,\dots,0)$ for all $d\in \mathbb{N}$,
		\item[(iv)] $K_{b_0}(b_1,\dots, b_k)$ whenever $k\leq b_0 + 1$,
		\item[(v)] $K_2(1,1,1,2,4)$, and $K_2(1,1,2,3,c)$ for all $c\equiv 0 \pmod 2$.	
	\end{itemize}
	\label{lem_K21126k_parts}
\end{lemma}

\begin{proof}
Part~{\em (i)} follows from Theorem \ref{thm:pw_tree_characterisation} and Observation \ref{obs_K_universal}.
Part~{\em (ii)} follows from {\em (i)}, Proposition \ref{prop_sequential_join}, and the fact that $K_2(a,b,c,d)$ is spanned by $K_1(a,b,c) + K_1 + K_d$.
The graph $K_0(0,\dots,d,\dots,0)$ is a complete graph of order $d$, from which {\em (iii)} follows.
Part~{\em (iv)} follows from an application of Observation \ref{obs_pw_is_pi_is_pg} to the traceable graph $K_{b_0}(b_1,\dots, b_k)$ for $k\leq b_0 + 1$.
Finally, {\em (v)} is proven in \cite{ravdgraphs_baudon}.
\end{proof}

We begin by finding a convenient infinite family of RP graphs with toughness $\frac{2}{5}$.

\begin{thm}
	For all $k\geq 0$, $k\in \mathbb{Z}$, the graph $K_2(1,1,2,6,k)$ is RP.
	\label{thm_k2_1126k}
\end{thm}

\begin{proof}
	We first prove that $G_k = K_2(1,1,2,6,k)$ is RP for all $k\in \{1,\dots, 10\}$. 
	Note that $n(G_k) = 12 + k$. 
	Thus, it suffices to prove that for all $\lambda \in \{1, \dots, \lfloor\frac{12+k}{2}\rfloor\}$, there is a partition $\{A,B\}$ of $V(G_k)$ such that $|A|=\lambda$ and $G_k[A]$, $G_k[B]$ are both RP.
	
 Table~\ref{table:parts6}-\ref{table:parts10} list all the (subgraphs induced by) partitions needed to show that $G_k$ is RP for $k\leq 10$. 
	All the subgraphs induced by the partitions are RP either by Lemma \ref{lem_K21126k_parts}, or by the previous cases.
	For example, the $|A|=5$ row of Table \ref{tab_k1} shows how to partition $V(G_1) = \{A,B\}$ so that $|A| = 5$ and $G_1[A]$, $G_1[B]$ are both RP (see Figure \ref{fig_g1_s5}).  
	
		To prove $G_k$ is RP for $k\geq 11$, we use induction. 
		Let $k\geq 11$, assume $G_k$ is RP for all $j < k$, and let $\lambda$ be any integer in $\{1,\dots, \lfloor \frac{12+k}{2}\rfloor\}$. 
		Then we can partition $V(G_k)$ into two parts $\{A,B\}$ where $|A|=\lambda$ by picking $A$ such that $G_k[A] = K_0(0,0,0,0,\lambda)\simeq K_\lambda$ and $G_k[B] = K_2(1,1,2,6,k-s)$.
		$G_k[A]$ is RP since it is a complete graph, and $G_k[B]$ is RP by induction, completing the proof.
\end{proof}
		
	\begin{figure}[h]
	\begin{center}
	\begin{tikzpicture}
		[scale=0.8,inner sep=0.6mm, 
		vertex/.style={circle,thick,draw}, 
		thickedge/.style={line width=2pt}] 
		
		\begin{scope}[shift={(0,0)}]
			\node[vertex] (x) at (-0.5,0) {};
			\node[vertex, fill=black] (y) at (0.5,0)  {};
			
			\node[vertex] (1) at (-4,3) {};
			\node[vertex] (2) at (-3,3) {};
			
			\node[vertex] (3) at (-2,3) {};	
			\node[vertex] (4) at (-1,3) {};
			
			\node[vertex, fill=black] (5) at (-0,3) {};
			\node[vertex, fill=black] (6) at (0.6,3.8) {};
			\node[vertex, fill=black] (7) at (1.5,4.2) {};
			\node[vertex, fill=black] (8) at (2.5,4.2) {};
			\node[vertex, fill=black] (9) at (3.4,3.8) {};
			\node[vertex, fill=black] (10) at (4,3) {};
			
			\node[vertex, fill=black] (11) at (5,3) {};
			
			\foreach \i in {1,...,11}
				{
					\draw[black!35] (x)--(\i)--(y);
				}
		
			\draw[black!35] (x)--(y);
			\draw (1)--(x)--(2) (3)--(x)--(4);
			\draw (3)--(4);
			
			\foreach \i in {5,...,11}
			{
				\draw[thick] (\i)--(y);
			}
		
			\foreach \i in {5,...,10}
			{
				\foreach \j in {\i,...,10}
				{
					\draw[thick] (\i)--(\j);
				}
			}
		\end{scope}

	\end{tikzpicture}
	\end{center}
	\caption{$V(G_1) = \{A, B\}$ where $|A|=5$, $G_1[A] = K_1(1,1,2,0,0)$ and $G_1[B] = K_1(0,0,0,6,1)$. The subgraph $G_1[B]$ is bolded, and the edges not belonging to either $G_1[A]$ or $G_1[B]$ are light grey.}
	\label{fig_g1_s5}
	\end{figure}
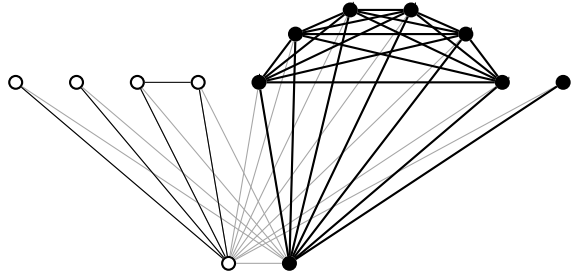
	
	\begin{table}[H]
		\centering
		\caption{Partitions of $G_1$ for $\lambda \leq \lfloor \frac{12+1}{2}\rfloor = 6$.}
		\label{tab_k1}
		\begin{tabular}{|l|l|l|l|l|l|l|}
			\cline{1-3} \cline{5-7}
			$\lambda$ & $G_1[A]$      & $G_1[B]$      &  & $\lambda$ & $G_1[A]$      & $G_1[B]$      \\ \cline{1-3} \cline{5-7} 
			1         & $K_0(0,0,0,0,1)$ & $K_2(1,1,2,6,0)$ &  & 2         & $K_0(0,0,2,0,0)$ & $K_2(1,1,0,6,1)$ \\ \cline{1-3} \cline{5-7} 
			3         & $K_1(1,1,0,0,0)$ & $K_1(0,0,2,6,1)$ &  & 4         & $K_1(1,0,2,0,0)$ & $K_1(0,1,0,6,1)$ \\ \cline{1-3} \cline{5-7} 
			5         & $K_1(1,1,2,0,0)$ & $K_1(0,0,0,6,1)$ &  & 6         & $K_0(0,0,0,6,0)$ & $K_2(1,1,2,0,1)$ \\ \cline{1-3} \cline{5-7} 
		\end{tabular}\label{table:parts6}
	\end{table}
	\begin{table}[H]
		\centering
		\caption{Partitions of $G_2$ for $\lambda \leq \lfloor \frac{12+2}{2}\rfloor = 7$.}
		\label{tab_k2}
		\begin{tabular}{|l|l|l|l|l|l|l|}
			\cline{1-3} \cline{5-7}
			$\lambda$ & $G_2[A]$            & $G_2[B]$              &  & $\lambda$ & $G_2[A]$      & $G_2[B]$      \\ \cline{1-3} \cline{5-7} 
			$\leq 2$  & $K_0(0,0,0,0,\lambda)$ & $K_2(1,1,2,6,2-\lambda)$ &  & 3         & $K_0(0,0,0,3,0)$ & $K_2(1,1,2,3,2)$ \\ \cline{1-3} \cline{5-7} 
			4         & $K_1(1,0,2,0,0)$       & $K_1(0,1,0,6,2)$         &  & 5         & $K_1(1,1,2,0,0)$ & $K_1(0,0,0,6,1)$ \\ \cline{1-3} \cline{5-7} 
			6         & $K_0(0,0,0,6,0)$       & $K_2(1,1,2,0,2)$         &  & 7         & $K_1(1,0,2,3,0)$ & $K_1(0,1,0,3,2)$ \\ \cline{1-3} \cline{5-7} 
		\end{tabular}
	\end{table}
	\begin{table}[H]
		\centering
		\caption{Partitions of $G_3$ for $\lambda \leq \lfloor \frac{12+3}{2}\rfloor = 7$.}
		\label{tab_k3}
		\begin{tabular}{|l|l|l|llll}
			\cline{1-3} \cline{5-7}
			$\lambda$ & $G_3[A]$            & $G_3[B]$              & \multicolumn{1}{l|}{} & \multicolumn{1}{l|}{$\lambda$} & \multicolumn{1}{l|}{$G_3[A]$}      & \multicolumn{1}{l|}{$G_3[B]$}      \\ \cline{1-3} \cline{5-7} 
			$\leq 3$  & $K_0(0,0,0,0,\lambda)$ & $K_2(1,1,2,6,3-\lambda)$ & \multicolumn{1}{l|}{} & \multicolumn{1}{l|}{4}         & \multicolumn{1}{l|}{$K_0(0,0,0,4,0)$} & \multicolumn{1}{l|}{$K_2(1,1,2,2,3)$} \\ \cline{1-3} \cline{5-7} 
			5         & $K_1(1,1,2,0,0)$       & $K_1(0,0,0,6,3)$         & \multicolumn{1}{l|}{} & \multicolumn{1}{l|}{6}         & \multicolumn{1}{l|}{$K_0(0,0,0,6,0)$} & \multicolumn{1}{l|}{$K_2(1,1,2,0,3)$} \\ \cline{1-3} \cline{5-7} 
			7         & $K_1(1,0,2,0,3)$       & $K_1(0,1,0,6,0)$         &                       &                                &                                       &                                       \\ \cline{1-3}
		\end{tabular}
	\end{table}
	\begin{table}[H]
		\centering
		\caption{Partitions of $G_4$ for $\lambda \leq \lfloor \frac{12+4}{2}\rfloor = 8$.}
		\label{tab_k4}
		\begin{tabular}{|l|l|l|llll}
			\cline{1-3} \cline{5-7}
			$\lambda$ & $G_4[A]$            & $G_4[B]$              & \multicolumn{1}{l|}{} & \multicolumn{1}{l|}{$\lambda$} & \multicolumn{1}{l|}{$G_4[A]$}      & \multicolumn{1}{l|}{$G_4[B]$}      \\ \cline{1-3} \cline{5-7} 
			$\leq 4$  & $K_0(0,0,0,0,\lambda)$ & $K_2(1,1,2,6,4-\lambda)$ & \multicolumn{1}{l|}{} & \multicolumn{1}{l|}{5}         & \multicolumn{1}{l|}{$K_1(1,1,2,0,0)$} & \multicolumn{1}{l|}{$K_1(0,0,0,6,4)$} \\ \cline{1-3} \cline{5-7} 
			6         & $K_0(0,0,0,6,0)$       & $K_2(1,1,2,0,4)$         & \multicolumn{1}{l|}{} & \multicolumn{1}{l|}{7}         & \multicolumn{1}{l|}{$K_1(1,0,2,0,3)$} & \multicolumn{1}{l|}{$K_1(0,1,0,6,1)$} \\ \cline{1-3} \cline{5-7} 
			8         & $K_1(1,0,2,0,4)$       & $K_1(0,1,0,6,0)$         &                       &                                &                                       &                                       \\ \cline{1-3}
		\end{tabular}
	\end{table}
	\begin{table}[H]
		\centering
		\caption{Partitions of $G_5$ for $\lambda \leq \lfloor \frac{12+5}{2}\rfloor = 8$.}
		\label{tab_k5}
		\begin{tabular}{|l|l|l|l|l|l|l|}
			\cline{1-3} \cline{5-7}
			$\lambda$ & $G_5[A]$            & $G_5[B]$              &  & $\lambda$ & $G_5[A]$      & $G_5[B]$      \\ \cline{1-3} \cline{5-7} 
			$\leq 5$  & $K_0(0,0,0,0,\lambda)$ & $K_2(1,1,2,6,5-\lambda)$ &  & 6         & $K_0(0,0,0,6,0)$ & $K_2(1,1,2,0,5)$ \\ \cline{1-3} \cline{5-7} 
			7         & $K_1(1,0,0,0,5)$       & $K_1(0,1,2,6,0)$         &  & 8         & $K_1(0,0,2,0,5)$ & $K_1(1,1,0,6,0)$ \\ \cline{1-3} \cline{5-7} 
		\end{tabular}
	\end{table}
	\begin{table}[H]
		\centering
		\caption{Partitions of $G_6$ for $\lambda \leq \lfloor \frac{12+6}{2}\rfloor = 9$.}
		\label{tab_k6}
		\begin{tabular}{|l|l|l|l|l|l|l|}
			\cline{1-3} \cline{5-7}
			$\lambda$ & $G_6[A]$            & $G_6[B]$              &  & $\lambda$ & $G_6[A]$      & $G_6[B]$      \\ \cline{1-3} \cline{5-7} 
			$\leq 6$  & $K_0(0,0,0,0,\lambda)$ & $K_2(1,1,2,6,6-\lambda)$ &  & 7         & $K_1(1,0,2,3,0)$ & $K_1(0,1,0,3,6)$ \\ \cline{1-3} \cline{5-7} 
			8         & $K_1(1,0,0,0,6)$       & $K_1(0,1,2,6,0)$         &  & 9         & $K_1(1,1,0,0,6)$ & $K_1(0,0,2,6,0)$ \\ \cline{1-3} \cline{5-7} 
		\end{tabular}
	\end{table}
	\begin{table}[H]
		\centering
		\caption{Partitions of $G_7$ for $\lambda \leq \lfloor \frac{12+7}{2}\rfloor = 9$.}
		\label{tab_k7}
		\begin{tabular}{|l|l|l|llll}
			\cline{1-3} \cline{5-7}
			$\lambda$ & $G_7[A]$            & $G_7[B]$              & \multicolumn{1}{l|}{} & \multicolumn{1}{l|}{$\lambda$} & \multicolumn{1}{l|}{$G_7[A]$}      & \multicolumn{1}{l|}{$G_7[B]$}      \\ \cline{1-3} \cline{5-7} 
			$\leq 7$  & $K_0(0,0,0,0,\lambda)$ & $K_2(1,1,2,6,7-\lambda)$ & \multicolumn{1}{l|}{} & \multicolumn{1}{l|}{8}         & \multicolumn{1}{l|}{$K_1(1,0,0,6,0)$} & \multicolumn{1}{l|}{$K_1(0,1,2,0,7)$} \\ \cline{1-3} \cline{5-7} 
			9         & $K_1(1,1,0,6,0)$       & $K_1(0,0,2,0,7)$         &                       &                                &                                       &                                       \\ \cline{1-3}
		\end{tabular}
	\end{table}
	\begin{table}[H]
		\centering
		\caption{Partitions of $G_8$ for $\lambda \leq \lfloor \frac{12+8}{2}\rfloor = 10$.}
		\label{tab_k8}
		\begin{tabular}{|l|l|l|llll}
			\cline{1-3} \cline{5-7}
			$\lambda$ & $G_8[A]$            & $G_8[B]$              & \multicolumn{1}{l|}{} & \multicolumn{1}{l|}{$\lambda$} & \multicolumn{1}{l|}{$G_8[A]$}      & \multicolumn{1}{l|}{$G_8[B]$}      \\ \cline{1-3} \cline{5-7} 
			$\leq 8$  & $K_0(0,0,0,0,\lambda)$ & $K_2(1,1,2,6,8-\lambda)$ & \multicolumn{1}{l|}{} & \multicolumn{1}{l|}{9}         & \multicolumn{1}{l|}{$K_1(1,1,0,6,0)$} & \multicolumn{1}{l|}{$K_1(0,0,2,0,8)$} \\ \cline{1-3} \cline{5-7} 
			10        & $K_1(1,0,2,6,0)$       & $K_1(0,1,0,0,8)$         &                       &                                &                                       &                                       \\ \cline{1-3}
		\end{tabular}
	\end{table}
	\begin{table}[H]
		\centering
		\caption{Partitions of $G_9$ for $\lambda \leq \lfloor \frac{12+9}{2}\rfloor = 10$.}
		\label{tab_k9}
		\begin{tabular}{|l|l|l|l|l|l|l|}
			\cline{1-3} \cline{5-7}
			$\lambda$ & $G_9[A]$            & $G_9[B]$              &  & $\lambda$ & $G_9[A]$      & $G_9[B]$      \\ \cline{1-3} \cline{5-7} 
			$\leq 9$  & $K_0(0,0,0,0,\lambda)$ & $K_2(1,1,2,6,9-\lambda)$ &  & 10        & $K_1(1,0,2,6,0)$ & $K_1(0,1,0,0,9)$ \\ \cline{1-3} \cline{5-7} 
		\end{tabular}
	\end{table}
	\begin{table}[H]
		\centering
		\caption{Partitions of $G_{10}$ for $\lambda \leq \lfloor \frac{12+9}{2}\rfloor = 10$.}
		\label{tab_k10}
		\begin{tabular}{|l|l|l|l|l|l|l|}
			\cline{1-3} \cline{5-7}
			$\lambda$ & $G_{10}[A]$            & $G_{10}[B]$               &  & $\lambda$ & $G_{10}[A]$      & $G_{10}[B]$      \\ \cline{1-3} \cline{5-7} 
			$\leq 10$ & $K_0(0,0,0,0,\lambda)$ & $K_2(1,1,2,6,10-\lambda)$ &  & 11        & $K_1(1,0,2,0,7)$ & $K_1(0,1,0,6,3)$ \\ \cline{1-3} \cline{5-7} 
		\end{tabular}\label{table:parts10}
	\end{table}

Using replacement graphs formed from the graphs $K_2(1,1,2,6,j)$, $j$ a positive integer, we can create arbitrarily large RP graphs with arbitrarily large cuts $S$ having $2|S|+1$ components.

\begin{corollary}
	For all $s\geq 1$, there exists an infinite family $\mathcal{G}_s$ of graphs such that each graph $G$ in $\mathcal{G}_s$ has a vertex cut $S$ with $|S|=s$ and $c(G-S) = 2s+1$.
	\label{cor_2kplus1_family}
\end{corollary}

\begin{proof}
	For $j$ a positive integer, let $H_1(j) = T(1,1,2j)$, and let $H_2(j) = K_2(1,1,2,6,j)$. These graphs are all RP by Theorems \ref{thm:pw_tree_characterisation} and \ref{thm_k2_1126k}.
	Define $H_{s+2}(j)$ inductively by setting
	\[
		H_{s+2}(j) = K_2 + (K_1 \cup K_1 \cup K_2 \cup K_6 \cup H_s(j))
	\]
	By Theorems~\ref{thm_replacement} and \ref{thm_k2_1126k}, the graph $H_{s+2}(j)$ is RP.
	It's clear that the graph $H_s(j)$ has a vertex cut $S$ with $|S|=s$ and $c(G-S) = 2s+1$ (for example, see Figure \ref{fig_h3_1}). 
	To complete the proof, we let $\mathcal{G}_s = \{H_s(j)\}_{j\in\mathbb{N}}$.
\end{proof}

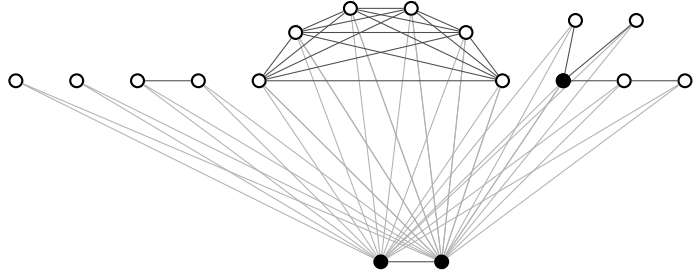
\begin{figure}[h]
	\begin{center}
	\begin{tikzpicture}
		[scale=0.8,inner sep=0.6mm, 
		vertex/.style={circle,thick,draw}, 
		thickedge/.style={line width=2pt}] 
		
		\begin{scope}[shift={(0,0)}]
			\node[vertex, fill=black] (x) at (2,0) {};
			\node[vertex, fill=black] (y) at (3,0)  {};
			
			\node[vertex] (1) at (-4,3) {};
			\node[vertex] (2) at (-3,3) {};
			
			\node[vertex] (3) at (-2,3) {};	
			\node[vertex] (4) at (-1,3) {};
			
			\node[vertex] (5) at (-0,3) {};
			\node[vertex] (6) at (0.6,3.8) {};
			\node[vertex] (7) at (1.5,4.2) {};
			\node[vertex] (8) at (2.5,4.2) {};
			\node[vertex] (9) at (3.4,3.8) {};
			\node[vertex] (10) at (4,3) {};
			
			\node[vertex, fill=black] (11) at (5,3) {};
			\node[vertex] (12) at (5.2,4) {};
			\node[vertex] (13) at (6.2,4) {};
			\node[vertex] (14) at (6,3) {};
			\node[vertex] (15) at (7,3) {};
			
			\foreach \i in {1,...,15}
				{
					\draw[black!30] (x)--(\i)--(y);
				}
		
			\draw[black!70] (x)--(y);
			\draw[black!70] (3)--(4);
			
			\foreach \i in {5,...,11}
			{
				\draw[black!30] (\i)--(y);
			}
		
			\foreach \i in {5,...,10}
			{
				\foreach \j in {\i,...,10}
				{
					\draw[black!70] (\i)--(\j);
				}
			}
		
			\draw[black!70] (11)--(12) (11)--(13) (11)--(14)--(15);
		
			\node[vertex, fill=black] (11) at (5,3) {};
		\end{scope}

	\end{tikzpicture}
\end{center}
\caption{The graph $H_3(1)$. The vertices of a cut set $S$ with $|S|=3$ and $c(H_3(1)-S) = 7$ are bolded.}
\label{fig_h3_1}
\end{figure}

\begin{lemma}
    The graphs $K_2(1,2,3,4,6)$ and $K_2(1,2,2,3,4)$ are RP.
    \label{lem_12346_12234}
\end{lemma}

\begin{proof}
    The semistar $K_2(1,2,3,4,6)$ has 18 vertices.
    Table \ref{tab_k2_12346} below shows that for all $\lambda \in \{1,\dots, 9\}$, the graph $K_2(1,2,3,4,6)$ has a partition $\{A,B\}$ such that both parts induce RP graphs and $|A| = \lambda$. The parts are RP by Lemma \ref{lem_K21126k_parts} and Theorem \ref{thm_k2_1126k}.
    
    \begin{table}[h]
\centering
\caption{Partitions of $G=K_2(1,2,3,4,6)$ for $\lambda \leq 9$.}
\label{tab_k2_12346}
\begin{tabular}{|l|l|l|l|l|l|l|}
\cline{1-3} \cline{5-7}
$\lambda$ & $G[A]$      & $G[B]$      &  & $\lambda$ & $G[A]$      & $G[B]$      \\ \cline{1-3} \cline{5-7} 
1         & $K_0(1,0,0,0,0)$ & $K_2(0,2,3,4,6)$ &  & 2         & $K_0(0,2,0,0,0)$ & $K_2(1,0,3,4,6)$ \\ \cline{1-3} \cline{5-7} 
3         & $K_0(0,0,3,0,0)$ & $K_2(1,2,0,4,6)$ &  & 4         & $K_0(0,0,0,4,0)$ & $K_2(1,2,3,0,6)$ \\ \cline{1-3} \cline{5-7} 
5         & $K_1(1,0,3,0,0)$ & $K_1(0,2,0,4,6)$ &  & 6         & $K_0(0,0,0,0,6)$ & $K_2(1,2,3,4,0)$ \\ \cline{1-3} \cline{5-7} 
7         & $K_1(1,2,3,0,0)$ & $K_1(0,0,0,4,6)$ &  & 8         & $K_1(0,0,3,4,0)$ & $K_1(1,2,0,0,6)$ \\ \cline{1-3} \cline{5-7} 
9         & $K_1(1,0,3,4,0)$ & $K_1(0,2,0,0,6)$ &  &           &                  &                  \\ \cline{1-3} \cline{5-7} 
\end{tabular}
\end{table}

The proof that the 14-vertex graph $K_2(1,2,2,3,4)$ is RP follows similarly by considering Table \ref{tab_k2_12234}.

\begin{table}[h]
\centering
\caption{Partitions of $G=K_2(1,2,2,3,4)$ for $\lambda \leq 7$.}
\label{tab_k2_12234}
\begin{tabular}{|l|l|l|l|l|l|l|}
\cline{1-3} \cline{5-7}
$\lambda$ & $G[A]$      & $G[B]$      &  & $\lambda$ & $G[A]$      & $G[B]$      \\ \cline{1-3} \cline{5-7} 
1         & $K_0(0,1,0,0,0)$ & $K_2(1,1,2,3,4)$ &  & 2         & $K_0(0,2,0,0,0)$ & $K_2(1,0,2,3,4)$ \\ \cline{1-3} \cline{5-7} 
3         & $K_0(0,0,0,3,0)$ & $K_2(1,2,2,0,4)$ &  & 4         & $K_0(0,0,0,0,4)$ & $K_2(1,2,2,3,0)$ \\ \cline{1-3} \cline{5-7} 
5         & $K_1(1,1,2,0,0)$ & $K_1(0,1,0,3,4)$ &  & 6         & $K_1(0,0,2,3,0)$ & $K_1(1,2,0,0,4)$ \\ \cline{1-3} \cline{5-7} 
7         & $K_1(0,0,2,0,4)$ & $K_1(1,2,0,3,0)$ &  &           &                  &                  \\ \cline{1-3} \cline{5-7} 
\end{tabular}
\end{table}

\end{proof}

\begin{thm}
    The semistar $K_3(1,1,1,2,2,3,4,6)$ is RP.
    \label{thm_S3_k8}
\end{thm}

\begin{proof}
    Let $G = K_3(1,1,1,2,2,3,4,6)$, and note that $n(G) = 23$.
    We show that for all $\lambda \leq 11$, the vertex set $V$ of $G$ has a partition $\{A, B\}$ such that $|A| = \lambda$, and the induced graphs $G[A]$ and $G[B]$ are RP.

    \bm{$\lambda = 1:$} Let $S_1 = K_1$ be a 1-vertex component of $G$, and $T_1 = K_3(1,1,2,2,3,4,6)$. By Theorem \ref{thm_replacement}, Lemma \ref{lem_K21126k_parts} and Lemma \ref{lem_12346_12234}, we can construct an RP spanning subgraph $H$ of $T_1$.
    $H$ is an RP replacement graph made using $K_1(1,6,14)$ and $K_2(1,2,2,3,4)$:
    \[
        T_1 \geq H = K_1 + \left(K_1 \cup K_6 \cup K_2(1,2,2,3,4) \right).
    \]
    
    \bm{$\lambda = 2:$} Let $S_2 = K_2$ a 2-vertex component of $G$, and $T_2 = K_3(1,1,1,2,3,4,6)$. By Theorem \ref{thm_k2_1126k}, the graph $K_2(1,1,2,6,9)$ is RP. Thus, we can construct an RP spanning subgraph $H$ of $T_2$ using $K_2(1,1,2,6,9)$ and $K_1(1,3,4)$:
    \[
        T_2 \geq H = K_2 + \left( K_1 \cup K_1 \cup K_2 \cup K_6 \cup K_1(1,3,4) \right).
    \]
    
    \bm{$\lambda = 3:$} Let $S_3 = K_3$ be the 3-vertex component, and $T_3 = K_3(1,1,1,2,2,4,6)$. Using $K_2(1,1,2,6,8)$ and $K_1(1,2,4)$, we construct an RP replacement graph $H$ that spans $T_3$:
    \[
        T_3 \geq H = K_2 + \left(K_1 \cup K_1 \cup K_2 \cup K_6 \cup K_1(1,2,4) \right).
    \]
    
    \bm{$\lambda = 4:$} Let $S_4 = K_4$ and $T_4 = K_3(1,1,1,2,2,3,6)$. We construct an RP spanning subgraph $H$ of $T_4$:
    \[
        T_4 \geq H = K_2 + \left(K_1 \cup K_1 \cup K_2 \cup K_6 \cup K_1(1,2,3) \right).
    \]
    
    \bm{$\lambda = 5:$} Let $S_5 = K_1(1,1,2)$ and $T_5 = K_2(1,2,3,4,6)$.
    
    \bm{$\lambda = 6:$} Let $S_6 = K_6$ and $T_6 = K_3(1,1,1,2,2,3,4)$. The graph $H$ below is an RP spanning subgraph of $T_6$, constructed using $K_2(1,1,2,3,8)$ and $T_1(1,2,4)$:
    \[
        T_6 \geq H = K_2 + \left(K_1 \cup K_1 \cup K_2 \cup K_3 \cup K_1(1,2,4) \right).
    \]
    
    \bm{$\lambda = 7:$} Let $S_7 = K_1(1,2,3)$ and $T_7 = K_2(1,1,2,4,6)$.
    
    \bm{$\lambda = 8:$} Let $S_8 = K_1(1,2,4)$ and $T_8 = K_2(1,1,2,3,6)$.
    
    \bm{$\lambda = 9:$} Let $S_9 = K_1(1,3,4)$ and $T_9 = K_2(1,1,2,2,6)$.
    
    \bm{$\lambda = 10:$} Let $S_{10} = K_1(1,2,6)$ and $T_{10} = K_2(1,1,2,3,4)$.
    
    \bm{$\lambda = 11:$} Let $S_{11} = K_2(1,1,2,2,3)$ and $T_{11} = K_1(1,4,6)$.
\end{proof}

\section{Minimal RP graphs}\label{sec:minimal}

Let $b_0, \dots, b_k$ be positive integers.
Call $K_{b_0}(b_1,\dots, b_k)$ a \textbf{minimal \bm{$(b_0,k)$} RP semistar} if there do not exist positive integers $c_1,\dots, c_k$ such that both the following hold:
\begin{itemize}
	\item $K_{b_0}(c_1,\dots, c_k)$ is RP, and
	\item $K_{b_0}(c_1,\dots, c_k)$ is a proper subgraph of $K_{b_0}(b_1,\dots, b_k)$.
\end{itemize}
It is easy to see that $K_1(1,1,2)$ is the unique minimal $(1,3)$ RP semistar. 
By Observation~\ref{obs_K_universal}, every RP graph $G$ with a cut-vertex $v$ such that $c(G-v) = 3$ has order 5 or more.
In this section, we show that $K_2(1,1,2,2,3)$ and $K_2(1,1,1,2,4)$ are the only minimal $(2,5)$ RP semistars.
Thus, every graph $G$ with a $2$-vertex cut $S$ such that $c(G-S) = 5$ has order $11$ or more.
Further, we show that the RP semistar $K_3(1,1,1,2,2,3,4,6)$ is minimal.

Let $\mathcal{G}(b_0,k) = \{K_{b_0}(b_1,\dots, b_k) : 1\leq b_1 \leq \dots \leq b_k\}$. 
The poset $\mathcal{G}(b_0,k)$ ordered by subgraph inclusion embeds into $\mathbb{N}^k$ (with the product order) in the obvious way. 
Dickson's Lemma states that the product $\mathbb{N}^k$ contains neither infinite anti-chains, nor infinite strictly descending sequences \cite{dickson_lemma}.
\begin{remark}
	For each pair $(b_0,k)$ of positive integers, there are finitely many minimal $(b_0, k)$ RP semistars.
	\label{rem:fin_min_semistars}
\end{remark}

A well known theorem of Tutte states that a graph $G$ has a perfect matching if and only if for every vertex cut $S$ of $G$, the graph $G-S$ has at most $|S|$ odd components \cite{match_tutte}.
The next lemma shows that this necessary condition can be generalised to partitions with connected parts of any size.

If $S$ is a finite set, then let $|S|_k$ denote the number $j$ in $\{0,1,\dots, k-1\}$ such that $|S| \equiv j \pmod k$. 
If $G$ is a graph, and $S\subseteq V(G)$, then let 
\[
	w_k(G,S) = \frac{1}{k-1}\cdot\sum\left\{|V(C)|_k : C \text{ a component of } G-S \right\}.
\]

The following result is given in \cite{avdconnectivity_baudon}.

\begin{lemma}\textup{\cite{avdconnectivity_baudon}}
	Let $G$ be a connected graph, $S$ a vertex cut of $G$ with $|S| < c(G-S)$, and $k\geq 2$ a positive integer.
	If $G$ is AP, then
	\[
	|S| + 1 \geq w_k(G,S).
	\]
	\label{lem_modcut_old}
\end{lemma}

We give a slight sharpening of this lemma.
The proof is similar to the proof in \cite{avdconnectivity_baudon}, with care taken to track the remainder term $\frac{|V(G)|_k}{k-1}$.

\begin{lemma}
	Let $G$ be a connected graph, and $S$ a subset of $V(G)$.
	If $G$ has a partition into connected parts $T_1, T_2, \dots, T_m$ such that $|T_i| = k$ for all $i\leq m-1$, and $|T_m| \leq k$, then
	\[
		|S| + \frac{|V(G)|_k}{k-1} \geq w_k(G,S).
	\]
	\label{lem_modcut}
\end{lemma}

\begin{proof}
	Note that either $|T_m| = k$ or $|T_m| = |V(G)|_k$.
	We begin by considering the following subgraph $G'$ of $G$:
	\[
		G' = \bigcup\{G[T_i] : T_i \cap S \neq \emptyset \} \cup G[S] \cup G[T_m].
	\]
	Observe that $S\subseteq V(G')$ and $|V(G)|_k = |V(G')|_k$.
	Further notice that the vertex set of each component of $G-S$ is a union of the vertex sets of components of $G'-S$, and possibly some of the sets $T_i$, $i < m$.
	Therefore, we get $w_k(G',S) \geq w_k(G,S)$.
	
	Consider the subgraph $G^* = G'-T_m$, and let $S^* = S \setminus T_m$.
	Since $|V(G^*)|_k = 0$, and $T_m$ has either $k$ vertices or $|V(G')|_k$ vertices, we obtain
	\begin{align*}
		\left(|S| + \frac{|V(G)|_k}{k-1}\right) - \left(|S^*| + \frac{|V(G^*)|_k}{k-1} \right) & = \left(|S| + \frac{|V(G')|_k}{k-1}\right) - \left(|S^*| + 0\right)\\ 
		&\geq \frac{|V(G')|_k}{k-1}\\
		&\geq w_k(G', S) - w_k(G^*, S^*)\\
		&\geq w_k(G, S) - w_k(G^*, S^*).
	\end{align*}

	To complete the proof, it suffices to show that $|S^*| \geq w_k(G^*, S^*)$.
	Each component of $G^*-S^*$ is of the form $T_i - S^*$ for some $i<m$, and each such $T_i$ has exactly $k$ vertices.
	Thus, we have 
	\[
		w_k(G^*, S^*) = \frac{|G^*-S^*|}{k-1}.
	\]
	
	Further, each vertex of $G^*$ is in some $T_i$, $i<m$. 
	The $T_i$ has at least one vertex of $S^*$ and at most $k-1$ vertices not in $S^*$. 
	Therefore, $|G^* - S^*| \leq (k-1)|S^*|$, so
	\[
		|S^*| \geq \frac{|G^*-S^*|}{k-1} = w_k(G^*, S^*),
	\]
	completing the proof.
\end{proof}

\begin{corollary}
	If $G$ is an AP (RP) graph with $S\subseteq V(G)$, and $k\geq 2$ a positive integer, then
	\[
	|S| + \frac{|V(G)|_k}{k-1} \geq w_k(G,S).
	\]
	\label{cor_modcut_lemma}
\end{corollary}

\begin{thm}
	The graphs $K_2(1,1,1,2,4)$ and $K_2(1,1,2,2,3)$ are the unique minimal $(2,5)$ RP semistars.
	\label{thm_2_5_minimal}
\end{thm}

\begin{proof}
	Let $G = K_2(b_1,\dots,b_5)$ be a minimal $(2,5)$ RP semistar with $b_1\leq \dots \leq b_5$.
	We can remove a single vertex of $G$ and still have an RP graph remaining.
	By minimality of $G$, and since no $(1,5)$ semistar is RP, we have $b_1=1$.
	Similarly, we can remove two vertices, so $b_i=2$ for some $i$.
	There are two possibilities for removing three vertices.
	
	\textit{Case 1:} The RP subgraph induced by the three removed vertices is $K_1(1,1)$, so $b_2 = 1$, and $G=K_2(1,1,2,s,t)$ for some positive integers $s\leq t$.
	If $s=1$, then $t\geq 4$ by Corollary \ref{cor_modcut_lemma} (consider $k=2$ and $k=3$). 
	Thus, if $s=1$, the single minimal RP semistar is $K_2(1,1,1,2,4)$.
	If $s=2$, then $t\geq 3$ since $K_2(1,1,1,2,2)$ and $K_2(1,1,2,2,2)$ are not RP by Corollary \ref{cor_modcut_lemma} with $k=3$. 
	So when $s=2$, the only minimal RP semistar is $K_2(1,1,2,2,3)$.
	When $s\geq 3$, we have $K_2(1,1,2,2,3) < K_2(1,1,2,s,t)$, and if $s\geq 4$, then $K_2(1,1,1,2,4) < K_2(1,1,2,s,t)$, which proves uniqueness in Case 1.
	
	\textit{Case 2:} The RP subgraph induced by the three removed vertices is a $K_3$, so $b_i = 3$ for some $i$.
	Thus, $G = K_2(1,2,3,s,t)$ for some $1\leq s \leq t$.
	An analysis similar to that in Case 1 shows that $K_2(1,1,2,2,3)$ is the only minimal RP semistar in Case 2.
\end{proof}

\begin{corollary}
	Let $G$ be an RP graph of order $n$.
	If $G$ has a cut $S$ with $|S|=2$ and $c(G-S) = 5$, then $n\geq 11$.
	\label{cor_2_5_order_11}
\end{corollary}

\begin{prop}
	The graph $K_3(1,1,1,2,2,3,4,6)$ is a minimal $(3,8)$ RP semistar.
	\label{prop_minimal_K3_11122346}
\end{prop}

\begin{proof}
	Let $G = K_3(1,1,1,2,2,3,4,6)$. 
	By Theorem \ref{thm_S3_k8}, this graph is RP.
	To prove minimality, it suffices to show that $G$ does not have an RP proper subgraph of the form $H = K_3(1,1,1,b_1,b_2, b_3, b_4, b_5)$, where $1\leq b_1\leq \dots \leq b_5$.
	Assume to the contrary that it does, and let $S$ be the 3-vertex cut of $H$.
	
	\textit{Case 1:} $b_1 = 1$. 
	If $b_1 = 1$, then $b_2,\dots, b_5$ are all even integers by Corollary \ref{cor_modcut_lemma}.
	Thus, $b_2 = b_3 = 2$ and $b_4\in \{2,4\}$.
	However, this contradicts Corollary \ref{cor_modcut_lemma} when $k=3$.
	
	\textit{Case 2:} $b_1 = 2$.
	Then $b_2 = 2$ and further, $b_3 = 3$.
	For if $b_3 = 2$, then $H$ is not RP per Corollary \ref{cor_modcut_lemma} (let $k=3$).
	Since $b_3 = 3$, $H$ has four odd components (the maximum number permitted by Corollary \ref{cor_modcut_lemma}), so $b_4$ and $b_5$ are even.
	Thus, $b_4 = 4$, and $b_5 \in \{4,6,8,\dots\}$.
	However, if $b_5=4$, then $w_5(H-S)> 4$, contradicting Corollary \ref{cor_modcut_lemma}.
	
	In either case, we derive a contradiction, so $G$ does not have such a subgraph $H$.
\end{proof}
 
\section{Bounding \bm{$c(G-S)$} from above}\label{sec:upbound}
In this section, we show that RP graphs are $\frac{1}{3}$-tough.
We have seen that there exist RP graphs with a cut-vertex $v$ such that $c(G-v)=3$. 
However, as we show in Theorem~\ref{thm:3toughnessupperbound}, for cuts $S$ of greater size in RP graphs, we must have $c(G-S)<3 |S|$, and this bound is sharp when $|S|=2$ or $|S|=3$.

We say that an RP graph $G$ of order $n$ is \emph{minimal with respect to $S$}, if there is no $(\lambda,n-\lambda)$-partition, for any $\lambda$, of $G$ into RP graphs $G_1$ and $G_2$ such that $G_1$ is a proper induced subgraph of any of the connected components of $G-S$. 

\begin{thm}\label{thm:3toughnessupperbound}
	Let $S$ be a cut of a graph $G$ with $|S|\geq 2$.
	If $c(G-S) \geq 3|S| $, then $G$ is not RP.
	\label{thm_c_geq_3s_plus_1}
\end{thm}

\begin{proof}
By Theorem \ref{thm_rp_balloon_toughness} and Observation \ref{obs_tripode_balloon_universal}, the result holds when $|S|=s=2$.
Further, per Theorem \ref{thm:pw_tree_characterisation}, if $|S|=1$, then $c(G-S) \leq 3$.
We proceed by strong induction, assuming the result holds for all integers $i$ such that $2\leq i < s$.
    
Suppose that $G$ is RP, and let $S$ be a cut in $G$, with $|S|=s\geq 3$. 
We can assume that $G$ is minimal with respect to $S$. 
Let $C_1,C_2,\dots,C_k$ be the connected components of $G-S$, with $|C_1|\leq |C_2|\leq \dots \leq |C_k|$. 
Suppose that $|C_k|=|C_{k-1}|+1$. 
Let $\lambda=|C_k|+1$ and find a $(\lambda,n-\lambda)$-partition of $G$ into RP graphs $G_1$ and $G_2$.
Since $|C_i|\notin \{\lambda,n-\lambda\}$ for every $1\leq i \leq k$, we must have that both $S_1=S\cap G_1$ and $S_2=S\cap G_2$ are non-empty. 
Furthermore, since $|S|\geq 3$, we cannot have that $|S_1|=|S_2|=1$. Therefore, by induction, we have
\[ c(G-S)\leq c(G_1-S_1)+c(G_2-S_2) \leq 3|S_1|+3|S_2|-1=3|S|-1. \]
Now, suppose that $|C_k|\neq|C_{k-1}|+1$. Let $\lambda=|C_{k-1}|+1$. 
Find a $(\lambda,n-\lambda)$-partition of $G$ into graphs $G_1$ and $G_2$, and recall that, by minimality, we cannot have $G_1 \leq C_k$. 
Then, a similar argument holds.
\end{proof}

As Theorems \ref{thm:pw_tree_characterisation} and \ref{thm_S3_k8} demonstrate, the bound in Theorem \ref{thm_c_geq_3s_plus_1} is sharp when $|S|\in \{2,3\}$.

\begin{corollary}
	Every RP graph is $\frac{1}{3}$-tough.
	\label{cor_pw_1_over_3_tough}
\end{corollary}

\section{Further Questions}\label{sec:quests}
We mention a few open questions.
\begin{enumerate}
    \item Consider all pairs $(G,S)$, where $G$ is an RP graph and $S$ is an $s$-vertex subset of $V(G)$, and let $\zeta(s) = \max\{c(G-S) : (G,S)\}$.
When $s>1$, Theorem \ref{thm_c_geq_3s_plus_1} and Corollary \ref{cor_2kplus1_family} show that $2s+1 \leq \zeta(k) \leq 3s-1$. 
Can either of these bounds be improved? 
Is the $3s-1$ upper bound sharp?
\item Is there some constant $c$ such that every $c$-tough graph is AP (RP)?
\item If $K_{b_0}(b_1, b_2, \dots, b_k)$ is RP, is the graph $K_{b_0}(b_1, \dots b_{i-1}, b_{i+1}, \dots, b_k)$ also RP for any $i \in \{1, 2, \ldots, k\}$?
\item In light of Remark \ref{rem:fin_min_semistars} and Proposition \ref{prop_minimal_K3_11122346}, $K_3(1,1,1,2,2,3,4,6)$ is one of finitely many minimal $(3,8)$ RP semistars. 
Are there others? 
If so, what are they?
\item Both minimal $(2,5)$ RP semistars are subgraphs of infinitely many $(2,5)$ RP semistars. For example, $K_2(1,1,2,2,3)$ is a subgraph of every $K_2(1,1,2,3,k)$ where $k\equiv 0 \pmod 2$ is positive, and $K_2(1,1,1,2,4)$ is a subgraph of $K_2(1,1,2,6,k)$ where $k\geq 1$.
Is $K_3(1,1,1,2,2,3,4,6)$ a subgraph of infinitely many $(3,8)$ RP semistars?
\item Do there exist pairs of positive integers $(b_0, k)$ for which there exist a finite, but positive number of $(b_0, k)$ RP semistars?
\end{enumerate}
 
\section*{Acknowledgements}

This work is supported by the National Science Foundation under NSF Award 2015425. Any opinions, findings, and conclusions or recommendations expressed in this material are those of the authors and do not necessarily reflect those of the National Science Foundation.

We thank the organisers and founders of the SAMSA-Masamu program, without whom this research would not have been possible. Special thanks to Peter Johnson and Hunter Rehm for their valuable input on the project.

\end{document}